\documentclass[11pt]{article}
\usepackage{times,amsmath,amsthm,amssymb}
\usepackage{url,float}
\usepackage{geometry}
\theoremstyle{plain}
\newtheorem{theorem}{Theorem}[section]
\newtheorem{lemma}[theorem]{Lemma}
\newtheorem{proposition}[theorem]{Proposition}
\newtheorem{corollary}[theorem]{Corollary}

\theoremstyle{definition}

\theoremstyle{remark}
\newtheorem{remark}[theorem]{Remark}
\newtheorem{example}[theorem]{Example}

\newcommand{\abk}{\allowbreak}

\newcommand{\BH}{\mathrm{BH}}
\newcommand{\Cr}{\mathrm{C}}
\newcommand{\D}{\mathrm{D}}
\newcommand{\GH}{\mathrm{GH}}

\renewcommand{\mod}{\hspace{4pt}\mathrm{mod}\hspace{4pt}}

\begin{document}

\title{\textbf{\Large{Classifying cocyclic Butson Hadamard matrices}}}

\author{
\textsc{Ronan Egan}
				\thanks{\textit{E-mail: r.egan3@nuigalway.ie}}\\
\textit{\footnotesize{National University of Ireland, Galway, Ireland}}\\
\textsc{Dane Flannery}
                \thanks{\textit{E-mail: dane.flannery@nuigalway.ie}}\\
\textit{\footnotesize{National University of Ireland, Galway, Ireland}} \\
\textsc{Padraig \'O Cath\'ain}
				\thanks{\textit{E-mail: p.ocathain@gmail.com}}\\\
\textit{\footnotesize{Monash University, Victoria 3800, Australia}}
}

\footnotetext{
This is the final form of this work.
No other version has been or will be submitted elsewhere.
}

\maketitle
\begin{center}
\begin{abstract}
\noindent
We classify all the cocyclic Butson Hadamard matrices
$\mathrm{BH}(n,p)$ of order $n$ over the $p$th roots of unity for
an odd prime $p$ and $np\leq 100$. That is, we compile a list of
matrices such that any cocyclic $\BH(n,p)$ for these $n$, $p$ is
equivalent to exactly one element in the list. Our approach
encompasses non-existence results and computational machinery for
Butson and generalized Hadamard matrices that are of independent
interest.
\end{abstract}

\end{center}

\vspace{0.3cm}



\clearpage

\section{Introduction} \label{Intr}
We present a new classification of Butson Hadamard matrices within
the framework of cocyclic design theory~\cite{deLF,Horadam}. New
non-existence results are also obtained. We extend {\sc
Magma}~\cite{Magma} and {\sf GAP}~\cite{GAP} procedures
implemented previously for $2$-cohomology and relative difference
sets~\cite{FlOB,POCRod,RDS} to determine the matrices and sort
them into equivalence classes.

Cocyclic development was introduced by de Launey and Horadam in
the 1990s, as a way of handling pairwise combinatorial designs
that exhibit a special symmetry. It has turned out to be a
powerful tool in the study of real Hadamard matrices (see
\cite{POCRod} for the most comprehensive classification). A basic
strategy, which we follow here, is to use algebraic and
cohomological techniques in systematically constructing the
designs.

Butson Hadamard matrices have applications in disparate areas such
as quantum physics and error-correcting codes. So lists of these
objects have value beyond design theory. We were motivated to
undertake the classification in this paper as a first step towards
augmenting the available data on complex Hadamard matrices (and we
did find several matrices not equivalent to any of those in the
online catalog \cite{CHLibrary}).

Specifically, we classify all Butson Hadamard matrices of order
$n$ over $p$th roots of unity for an odd prime $p$ and $np\leq
100$. The restriction to $p$th roots is a convenience that renders
each matrix generalized Hadamard over a cyclic group of order $p$;
for these we have a correspondence with central relative
difference sets that enables us to push the computation to larger
orders. It must be emphasized that most of the techniques that we
present apply with equal validity to generalized Hadamard matrices
over any abelian group---but are not valid for Butson Hadamard
matrices over $k$th roots of unity with $k$ composite. Moreover,
the tractability of the problem considered in this paper suggests
avenues for investigation of other cocyclic designs, such as
complex weighing matrices and orthogonal designs.

The paper is organized as follows. In Section~\ref{Background} we
set out background from design theory: key definitions, our
understanding of equivalence, and general non-existence results.
Section~\ref{EquivalenceTesting} is devoted to an explanation of
our algorithm to check whether two Butson Hadamard matrices are
equivalent. We recall the necessary essentials of cocyclic
development in Section~\ref{RehashCocyclic}. Then in
Section~\ref{CocyclicBH} we specialize to cocyclic Butson Hadamard
matrices. The full classification is outlined in
Section~\ref{FullClassification}. We end the paper with some
miscellaneous comments prompted by the classification.

For space reasons, the listing of matrices in our classification
is not given herein. It may be accessed at \cite{WWW}.

\section{Background}\label{Background}
Throughout, $p$ is a prime and $G$, $K$ are finite non-trivial
groups. We write $\zeta_k$ for $e^{2\pi {\rm i}/k}$.

\subsection{Butson and generalized Hadamard matrices}
A \textit{Butson Hadamard matrix of order $n$ and phase $k$},
denoted $\mathrm{BH}(n,k)$, is an $n\times n$ matrix $H$ with
entries in $\langle \zeta_k\rangle$ such that $HH^* = nI_n$ over
$\mathbb{C}$. Here $H^*$ is the usual Hermitian, i.e., complex
conjugate transpose.

For $n$ divisible by $|K|$, a \emph{generalized Hadamard matrix
$\GH(n,K)$ of order $n$ over $K$} is an $n\times n$ matrix
$H=[h_{ij}]$ whose entries $h_{ij}$ lie in $K$ and such that
\[
HH^*=nI_n+\frac{n}{|K|}\hspace*{.175mm} ({\textstyle \sum}_{x\in
K}x)(J_n-I_n)
\]
where $H^*=[h_{ji}^{-1}]$, $J_n$ is the all $1$s matrix, and the
matrix operations are performed over the group ring $\mathbb{Z}
K$.

The transpose of a $\BH(n,k)$ is a $\BH(n,k)$; the transpose of a
$\GH(n,K)$ is not necessarily a $\GH(n,K)$, except when $K$ is
abelian~\cite[Theorem 2.10.7] {deLF}. However, if $H$ is a Butson
or generalized Hadamard matrix then so too is $H^*$.

For the next couple of results, see Theorem~2.8.4 and Lemma 2.8.5
in \cite{deLF} (the former requires a theorem from
\cite{LamLeung}).
\begin{theorem}\label{Lamlike}
If there exists a $\BH(n,k)$, and $p_1,\dots,p_r$ are the primes
dividing $k$, then there exist $a_1,\dots,a_r\in \mathbb{N}$
such that $n=a_1p_1+\dots+a_rp_r$.
\end{theorem}

One consequence of Theorem~\ref{Lamlike} is that $\BH(n,p^t)$ can
exist only if $p  \hspace{1pt} | \hspace{1pt}  n$.

\begin{lemma}\label{WhenASumOfPthRootsIsZero}
Let $\omega$ be a primitive $p$th root of unity. Then
$\sum_{i=0}^n a_i\omega^i=0$ for $n<\abk p$ and $a_0, \ldots ,
a_n\in \mathbb{N}$ not all zero if and only if $n=p-1$ and $a_0 =
\cdots = a_n$.
\end{lemma}

Let $C=\langle x \rangle\cong \Cr_k$ and define $\eta_k:
\mathbb{Z} C \rightarrow \mathbb{Z}[\zeta_k]$ by
$\eta_k\big(\sum_{i=0}^{k-1}c_ix^i\big)=
\sum_{i=0}^{k-1}c_i\zeta_k^i$. The map $\eta_k$ extends to a ring
epimorphism $\mathrm{Mat}(n,\mathbb{Z} C) \rightarrow
\mathrm{Mat}(n,\mathbb{Z} [\zeta_k])$.
\begin{lemma}\label{OneWay}
If $M$ is a $\mathrm{GH}(n,\Cr_k)$ then $\eta_k(M)$ is a
$\mathrm{BH}(n,k)$; if $M$ is a $\BH(n, p)$ then $\eta_p^{-1}(M)$
is a $\mathrm{GH}(n,\Cr_p)$.
\end{lemma}
\begin{proof}
The first part is easy, and the second uses
Lemma~\ref{WhenASumOfPthRootsIsZero}. \qedhere
\end{proof}

Thus, a $\BH(n,p)$ is the same design as a $\GH(n,\Cr_p)$.
Butson's seminal paper~\cite{Butson} supplies a construction of
$\BH(2^ap^b,p)$ for $0\leq a \leq \abk b$.
\begin{example}
For composite $n$, the Fourier matrix (more properly, Discrete
Fourier Transform matrix) of order $n$ is a $\BH(n,n)$ but not a
$\GH(n,\Cr_n)$.
\end{example}

\begin{example}
There are no known examples of $\GH(n,K)$  when $K$ is not a
$p$-group. Indeed, finding a $\GH(n, K)$ with $|K|=n$ not a power
of $p$ would resolve a long-standing open problem in finite
geometry; namely, whether a finite projective plane always has
prime-power order.
\end{example}

\subsection{Equivalence relations}
Let $X$, $Y$ be $\GH(n,K)$s. We say that $X$ and $Y$ are
\emph{equivalent} if $MXN = Y$ for monomial matrices $M$, $N$ with
non-zero entries in $K$. If $X$, $Y$ are $\BH(n,k)$s then they are
equivalent if $MXN = Y$ for monomials $M$, $N$ with non-zero
entries from  $\langle \zeta_k\rangle$. Equivalence in either
situation is denoted $X\approx Y$, whereas if $M$, $N$ are
permutation matrices then $X$, $Y$ are \emph{permutation
equivalent} and we write $X\sim Y$. The equivalence operations
defined above are \emph{local}, insofar as they are applied
entrywise to a single row or column one at a time. We will not
regard taking the transpose or Hermitian as equivalence
operations.

If $H$ is a $\GH(n,K)$ then $H\approx H'$ where $H'$ is normalized
(its first row and column are all $1$s) and thus
\emph{row-balanced}: each element of $K$ appears with the same
frequency, $n/|K|$, in each non-initial row. Similarly, $H'$ is
column-balanced. Unless $k$ is prime, neither property is
necessarily held by a normalized $\BH(n,k)$.

\subsection{Non-existence of generalized Hadamard matrices}
Certain number-theoretic conditions exclude various odd $n$ as the
order of a generalized Hadamard matrix; see, e.g.,
\cite{CookeHeng,deL84b,Winterhof}. The main general result of this
kind that we need is due to de Launey~\cite{deL84b}.
\begin{theorem}
\label{deLSquareFree} Let $K$ be abelian, and $r$, $n$ be odd,
where $r$ is a prime dividing $|K|$. If a $\GH(n,K)$ exists then
every integer $m \not\equiv 0 \mod r$ that divides the square-free
part of $n$ has odd multiplicative order modulo $r$.
\end{theorem}

\begin{remark}\label{deL84Cor}
$\BH(n,p)$ do not exist for $(n,p) \in \{ (15,3), (33,3), (15,5)
\}$.
\end{remark}

We shall derive non-existence conditions for cocyclic $\BH(n,p)$
later.

\section{Deciding equivalence of Butson Hadamard matrices}
\label{EquivalenceTesting}

In this section we give an algorithm to decide equivalence of
Butson Hadamard matrices. The problem is reduced to deciding graph
isomorphism, which we carry out using \textit{Nauty}~\cite{nauty};
and subgroup conjugacy and intersection problems, routines for
which are available in \textsc{Magma}.

\subsection{Automorphism groups, the expanded design, and the
associated design}\label{AutExpAssoc}

The direct product $\mathrm{Mon}(n,\langle \zeta_k\rangle)\times
\mathrm{Mon}(n,\langle \zeta_k\rangle)$ of monomial matrix groups
acts on the (presumably non-empty) set of $\BH(n,k)$ via $(M,N)H =
MHN^*$. The orbit of $H$ is its equivalence class; the stabilizer
is its full \emph{automorphism group} $\mathrm{Aut}(H)$.
\begin{example}\label{Drakeetc}
(\cite[Section 9.2]{deLF}.) Denote the $r$-dimensional
$\mathrm{GF}(p)$-space by $V$. Then $D=[\hspace{1pt}
xy^{\!\top}\hspace{.5pt}]_{x,y\in V}$ is a $\GH(p^{r},\Cr_p)$,
written additively. In fact $D$ is the $r$-fold Kronecker product
of the Fourier matrix of order $p$ (so when $p=2$ we get the
Sylvester matrix). If $r\neq 1$ or $p>2$ then
$\mathrm{Aut}(D)\cong (\Cr_p \times \Cr_p^{r}) \rtimes
\mathrm{A\hspace*{-.16mm} GL}(r,p)$.
\end{example}

Let $\mathrm{Perm}(n)$ be the group of all $n\times n$ permutation
matrices. The \emph{permutation automorphism group}
$\mathrm{PAut}(X)$ of an $n\times n$ array $X$ consists of all
pairs $(P,Q)\in \mathrm{Perm}(n)^2$ such that $PXQ^\top = X$.
Clearly $\mathrm{PAut}(H)\leq \mathrm{Aut}(H)$. The array $X$ is
\emph{group-developed} over a group $G$ of order $n$ if $X \sim
[h(xy)]_{x,y\in G}$ for some map $h$. We readily prove that $X$ is
group-developed over $G$ if and only if $G$ is isomorphic to a
regular subgroup (i.e., subgroup acting regularly in its induced
actions on the sets of row and column indices) of
$\mathrm{PAut}(X)$.

The full automorphism group $\mathrm{Aut}(H)$ has no direct
actions on rows or columns of $H$. Rather, it acts on the
\emph{expanded design} $\mathcal{E}_H= [\zeta_k^{i+j}H]$ via a
certain isomorphism $\Theta$ of $\mathrm{Aut}(H)$ onto
$\mathrm{PAut}(\mathcal{E}_H)$: see~\cite[Theorem~9.6.12]{deLF}.
\begin{proposition}[Corollary 9.6.10, \cite{deLF}]
\label{LambdaEquivalenceImplies} If $H_1$ and $H_2$ are equivalent
$\BH(n,k)$s then $\mathcal{E}_{H_1}\sim \mathcal{E}_{H_2}$;
therefore $\mathrm{PAut}(\mathcal{E}_{H_1})$ and
$\mathrm{PAut}(\mathcal{E}_{H_2})$ are isomorphic as conjugate
subgroups of $\, \mathrm{Perm}(nk)^2$.
\end{proposition}

A converse of Proposition~\ref{LambdaEquivalenceImplies} also
holds, which we might use as a criterion to distinguish Butson
Hadamard matrices. For computational purposes it is preferable to
work with the $(0,1)$-matrix $\mathrm{A}_{H}$ (the
\emph{associated design} of $H$) obtained from $\mathcal{E}_H$ by
setting its non-identity entries to zero. Then we need an analog
of Proposition~\ref{LambdaEquivalenceImplies} for the associated
design. Before stating this, we say a bit more about the embedding
$\Theta:\abk \mathrm{Mon}(n,\langle \zeta_k\rangle)^2\rightarrow
\mathrm{Perm}(nk)^2$. It maps $(P,Q)$ to $(\theta^{(1)}(P),
\theta^{(2)}(Q))$ where $\theta^{(1)}$ (resp.~$\theta^{(2)}$)
replaces each non-zero entry by the permutation matrix
representing that entry in the right (resp.~left)  regular action
of $\langle \zeta_k\rangle$ on itself. Denote the image of
$\mathrm{Mon}(n,\langle \zeta_k\rangle)^2$ under $\Theta$ by
$M(n,k)$.
\begin{proposition}\label{AssocEquiv}
Let $H_1$, $H_2$ be $\BH(n,k)$s. We have $H_1\approx H_2$ if and
only if $\mathrm{A}_{H_1} = \abk X\mathrm{A}_{H_2}Y^\top$ for some
$(X, Y)\in M(n,k)$.
\end{proposition}
\begin{proof}
Suppose that $\theta^{(1)}(P)\mathrm{A}_{H_2}\theta^{(2)}(Q)^\top
= \mathrm{A}_{H_1}$, and write
$\mathcal{E}_{H_i}=\sum_{r\in\langle \zeta_k\rangle} rH_{i,r}$ (so
$\mathrm{A}_{H_i}=H_{i,1})$. By Theorem~9.6.7 and Lemma~9.8.3 of
\cite{deLF},
\[
H_{1,r} = \abk \theta^{(1)}(P)\abk H_{2,r}\theta^{(2)}(Q)^\top.
\]
Therefore $\mathcal{E}_{H_1} = \mathcal{E}_{PH_2Q^*}$ by
\cite[Lemma 9.6.8]{deLF}. This implies that $H_1=PH_2Q^*$. \qedhere
\end{proof}

We also use the following simple fact.
\begin{lemma}\label{Cosets}
Let $A$, $B$ be subgroups and $x$, $y$ be elements of a group $G$.
Then either $xA \cap yB=\emptyset$, or $xA\cap yB = g(A\cap B)$
for some $g\in G$.
\end{lemma}

We now state our algorithm to decide equivalence of Butson
Hadamard matrices $H_1$ and $H_2$ of order $n$ and phase $k$.

\begin{enumerate}
\item Compute $G_{1} = \mathrm{PAut}(\mathrm{A}_{H_{1}})$ with
\textit{Nauty}. \item Attempt to find $\sigma\in
\mathrm{Perm}(nk)^2$ such that $\sigma
\mathrm{A}_{H_1}=\mathrm{A}_{H_2}$.

If no such $\sigma$ exists then return $\tt{false}$. \item Compute
$U = G_{1} \cap M(n,k)$ and a transversal $T$ for $U$ in $G_{1}$.
\item If there exists $t \in T$ such that $\sigma t \in M(n,k)$
then return $\tt{true}$;

else return $\tt{false}$.
\end{enumerate}

\noindent If $H_1\approx H_2$ then $\sigma G_1\cap \abk M(n,k)\neq
\emptyset$ by Proposition~\ref{AssocEquiv}, so by
Lemma~\ref{Cosets} we must find a $t$ as in step 4. A report of
$\tt false$ is then correct by Proposition~\ref{AssocEquiv}; a
report of $\tt true$ is clearly correct. Note that if the
algorithm returns $\tt true$ then we find an element
$\Theta^{-1}(\sigma t)$ mapping $H_1$ to $H_2$.

Step 1 is a potential bottleneck, although it remains feasible for
graphs with several hundred vertices. Equivalence testing is
therefore practicable for many $\BH(n,k)$ that have been
considered in the literature.
\begin{example}
The authors of \cite{McNW} construct a series of $\BH(2p,p)$ but
cannot decide whether their matrices are equivalent to those of
Butson~\cite[Theorem 3.5]{Butson}. Our method, which has been
implemented in \textsc{Magma}, shows that the $\BH(10,5)$ denoted
$S_{10}$ in \cite{McNW} is equivalent to Butson's matrix in less
than $0.1$s (an explicit equivalence is given at \cite{WWW}).
\end{example}

\section{Cocyclic development}\label{RehashCocyclic}
Since our main concern is Butson Hadamard matrices, we recap the
essential ideas of cocyclic development solely for this type of
design.

\subsection{Second cohomology and designs}
Let $H$ be a $\BH(n,k)$, and let $W$ be the $k\times k$ block
circulant matrix with first row $(0_n, \abk \ldots , \abk 0_n ,
I_n)$. A regular subgroup of $\mathrm{PAut}(\mathcal{E}_H)$
containing the central element $(W^\top,W)$ is \emph{centrally
regular}. By \cite[Theorem 14.7.1]{deLF},
$\mathrm{PAut}(\mathcal{E}_H)$ has a centrally regular subgroup if
and only if $H\approx [\psi(x,y)]_{x, y \in G}$ for some $G$ and
cocycle $\psi:\abk G\times G \rightarrow \abk \langle
\zeta_k\rangle$; i.e., $\psi(x,y)\psi(xy,z)=\psi(x,yz)\psi(y,z)$
$\forall \ x, \abk y, \abk z \in G$. We say that $H\approx
[\psi(x,y)]_{x, y \in G}$ is \emph{cocyclic}, with \emph{indexing
group $G$ and cocycle $\psi$}. A cocycle of $H$ is
\emph{orthogonal}.

Let $U$ be a finite abelian group and denote the group of all
cocycles $\psi:G\times G \rightarrow\abk  U$ by $Z(G,U)$. Our
cocycles are normalized, meaning that $\psi(x,y) = 1$ when $x$ or
$y$ is $1$. If $\phi:G\rightarrow U$ is a normalized map then
$\partial \phi\in Z(G,U)$ defined by $\partial
\phi(x,y)=\phi(x)^{-1}\phi(y)^{-1}\phi(xy)$ is a
\emph{coboundary}. These form a subgroup $B(G,U)$ of $Z(G,U)$, and
$H(G,U)=Z(G,U)/B(G,U)$ is the \emph{second cohomology group of
$G$}.

For each $\psi \in Z(G,U)$, the central extension $E(\psi)$ of $U$
by $G$ is the group with elements $\{ (g,u) \hspace{2.5pt} |
\hspace{2.5pt} \abk g\in G, \, u\in U\}$ and multiplication given
by $(g_1,u_1)(g_2,u_2)=(g_1g_2, \abk u_1u_2\psi(g_1,g_2))$.
Conversely, let  $E$ be a central extension of $U$ by $G$, with
embedding $\iota : U \rightarrow E$ and epimorphism $\pi : E
\rightarrow G$ satisfying $\ker \, \pi = \abk \iota (U)$. Choose a
normalized map $\tau: G\rightarrow E$ such that $\pi\tau =
\mathrm{id}_G$. Then $\psi_\tau (x,y)
=\iota^{-1}(\tau(x)\tau(y)\tau(xy)^{-1})$ defines a cocycle
$\psi_\tau$, and $E(\psi_\tau)\cong E$. Different choices of right
inverse $\tau$ of $\pi$ do not alter the cohomology class of
$\psi_\tau$.

A $\BH(n,k)$, $H$, is cocyclic with cocycle $\psi$ if and only if
$E(\psi)$ is isomorphic to a centrally regular subgroup of
$\mathrm{PAut}(\mathcal{E}_H)$ by an isomorphism mapping
$(1,\zeta_k)$ to $(W^\top,W)$. If $H$ is group-developed over $G$
then $H$ is equivalent to a cocyclic $\BH(n,k)$ with cocycle
$\psi\in B(G, \langle \zeta_k\rangle)$ and extension group
$E(\psi)\cong  G\times \Cr_k$.
\begin{example}\label{DrakeetcContinued}
The Butson Hadamard matrix $D$ in Example~\ref{Drakeetc} is
cocyclic, with indexing group $\Cr_p^r$ and cocycle $\psi\not \in
B(\Cr_p^r,\Cr_p)$ defined by $\psi(x,y)=xy^\top$. Note that $\psi$
is multiplicative and symmetric. If $p$ is odd then $E(\psi) \cong
\abk \Cr_p^{r+1}$. The determination of all cocycles, indexing
groups, and extension groups of $D$ would be an interesting
exercise; cf.~the account for $p=2$ in \cite[Chapter~21]{deLF}.
\end{example}

\subsection{Computing cocycles}\label{CCocycles}
We compute $Z(G,\langle \zeta_k\rangle)$ by means of the Universal
Coefficient theorem:
\[
H(G,U)= I(G,U)/B(G,U) \times T(G,U)/B(G,U)
\]
where $T(G,U)/B(G,U) \cong \mathrm{Hom}(H_2(G),U)$ and 
$I(G,U)/B(G,U)$ is the isomorphic image under inflation of
$\mathrm{Ext}(G/G',U)$. Here $G'=[G,G]$ and $H_2(G)$ is the Schur
multiplier of $G$.

We describe the calculation of $I(G,U)$ for $U=\abk \langle u
\rangle\cong \Cr_p$ as this is used in a later proof. Let
$\prod_i\langle g_iG'\rangle$ be the Sylow $p$-subgroup of $G/G'$,
where $|g_iG'|=p^{e_i}$. Define $M_i$ to be the $p^{e_i}\times
p^{e_i}$ matrix whose $r$th row is $(1,  \dots , 1, u , \dots ,
u)$, the first $u$ occurring in column $p^{e_i}-r+2$. Let $N_i$ be
the $|G|\times |G|$ matrix obtained by taking the Kronecker
product of $M_i$ with the all $1$s matrix. Up to permutation
equivalence, the $N_i$ constitute a complete set of
representatives for the elements of $I(G,U)/B(G,U)$ displayed as
cocyclic matrices. For more detail see \cite{FlOB}.

\subsection{Shift action}\label{ShiftAction}

In a search for orthogonal elements of $Z(G,\Cr_p)$, it is not
enough to test a single $\psi$ from each cohomology class $[\mu
]\in H(G,\Cr_p)$: if $\psi$ is orthogonal then $\psi'\in [\mu]$
need not be orthogonal. Horadam~\cite[Chapter~8]{Horadam}
discovered an action of $G$ on each $[\psi ]$ that preserves
orthogonality, defined by $\psi \cdot g = \psi
\partial (\psi_g)$ where $\psi_g(x)=\psi(g,x)$.
This `shift' action induces a linear representation $G\rightarrow
\mathrm{GL}(V)$ where $V$ is any $G$-invariant subgroup of
$Z(G,\Cr_p)$, allowing effective computation of orbits in
$V$~\cite{SR2014}.

\subsection{Further equivalences for cocyclic matrices}

Equivalence operations preserving cocycle orthogonality, apart
from local ones, arise from the shift action or natural actions on
$Z(G,\langle \zeta_p\rangle)$ by $\mathrm{Aut}(G)\times
\mathrm{Aut}(\Cr_p)$. The action by $\mathrm{Aut}(\Cr_p)$ alone
furnishes a \emph{global} equivalence operation. Together with the
local operations these generate the holomorph $\Cr_p\rtimes
\Cr_{p-1}$ of $\langle \zeta_p\rangle$~\cite[Theorem
4.4.10]{deLF}.

\subsection{Central relative difference sets}\label{BriefCRDS}

\begin{theorem}\label{FamousCRDS}
There exists a cocyclic $\BH(n,p)$ with cocycle $\psi$ if and only
if there is a relative difference set in $E(\psi)$ with parameters
$(n,p,n,n/p)$ and central forbidden subgroup $\langle (1,\zeta_p)
\rangle$.
\end{theorem}
\begin{proof}
This follows from \cite[Corollary 15.4.2]{deLF} or \cite[Theorem
4.1]{PereraHoradam}. \qedhere
\end{proof}

We explain one direction of the correspondence in
Theorem~\ref{FamousCRDS}. Let $E$ be a central extension of
$U\cong \Cr_p$ by $G$. Say $\iota$ embeds $U$ into the center of
$E$, and $\pi: E\rightarrow G$ is an epimorphism with kernel
$\iota(U)$. Suppose that $R = \{ d_1=1, d_2, \ldots,
d_n\}\subseteq E$ is an $(n,p,n,n/p)$-relative difference set with
forbidden subgroup $U$; i.e., the multiset of quotients
$d_id_{j}^{-1}$ for $j\neq i$ contains each element of $E\setminus
\iota(U)$ exactly $n/p$ times, and contains no element of
$\iota(U)$. Since $R$ is a transversal for the cosets of
$\iota(U)$ in $E$, we have $G = \{ g_i := \pi(d_i) \hspace{2.5pt}
| \hspace{2.5pt} 1\leq i\leq n \}$. Put $\tau(g_i)=d_i$. Then
$[\psi_\tau(x,y)]_{x,y\in G}$ is a $\BH(n,p)$.

\section{Cocyclic Butson Hadamard matrices}\label{CocyclicBH}
\begin{theorem}\label{NickelRedux}
Let $K$ be abelian, $n=|G|$ be divisible by $|K|$, $\psi \in
Z(G,K)$, and $H=\abk [\psi(x,y)]_{x,y\in G}$. Then $H$ is a
$\GH(n,K)$ if and only if it is row-balanced. In that event $H$ is
column-balanced too.
\end{theorem}
\begin{proof}
This follows from \cite[Lemma~6.6]{Horadam}, which generalizes a
phenomenon observed for cocyclic Hadamard
matrices~\cite[Theorem~16.2.1]{deLF}. \qedhere
\end{proof}

So we begin our classification by searching for balanced cocycles
in the relevant $Z(G,\Cr_p)$. When $k$ is not prime, a cocyclic
$\BH(n,k)$ need not be balanced; by \cite[Lemma~6.6]{Horadam}
again, $[\psi(x,y)]_{x,y\in G}$ for $\psi \in Z(G,\langle
\zeta_k\rangle)$ is a $\BH(n,k)$ if and only if each non-initial
row sum is zero.

We mention extra pertinent facts about Fourier matrices.
\begin{lemma}\label{Vandermonde}
The Fourier matrix of order $n$ is a cocyclic $\BH(n,n)$ with
indexing group $\Cr_n$. If $n$ is odd then it is equivalent to a
group-developed matrix.
\end{lemma}

\begin{proposition}[\cite{HiranandaniSchlenker}]
\label{CircppFourier} Every circulant $\BH(p,p)$ is equivalent to
the Fourier matrix of order $p$.
\end{proposition}

\begin{proposition}[\cite{Hirasakaetal}]\label{HiraSmall}
For $p\leq 17$, the Fourier matrix of order $p$ is the unique
$\BH(p,p)$ up to equivalence.
\end{proposition}

\subsection{Non-existence of cocyclic Butson Hadamard matrices}

As we expect, there are restrictions on the order of a
group-developed Butson Hadamard matrix.
\begin{lemma}\label{GeneralGDevLem}
Set $r_{j} = \mathrm{Re}(\zeta_k^j)$ and $s_{j} =
\mathrm{Im}(\zeta_k^j)$. A $\BH(n,k)$ with constant row and column
sums exists only if there are $x_{0},\ldots, x_{k-1} \in
\{0,1,\ldots,n\}$ satisfying
\begin{equation}\label{GeneralGDevPoly}
\big(\textstyle{\sum}_{j=0}^{k-1} r_{j}x_{j}\big)^2 +
\big(\textstyle{\sum}_{j=0}^{k-1}s_{j}x_{j}\big)^2 = n
\end{equation}
and $\sum_{j=0}^{k-1} x_{j} = n$.
\end{lemma}
\begin{proof}
Let $H$ be a $\BH(n,k)$ with every row and column summing to $s =
\sum_{j=0}^{k-1}x_{j}\zeta_k^j = a + bi$. Then
\[
nJ_n= J_nHH^*= s J_nH^* = s\overline{s}J_n
\]
implies $n=a^2 + b^2$, which is \eqref{GeneralGDevPoly}. \qedhere
\end{proof}
\begin{remark}\label{GDExclusions}
If $k = 2$ then \eqref{GeneralGDevPoly} just gives that $n$ must
be square, which is well-known. If $k=4$ then $n$ is the sum of
two integer squares. As a sample of other exclusions, the
following cannot be the order of a group-developed $\BH(n,k)$.
\begin{itemize}
\item[{\rm (i)}] $k=3$, $n\leq 100$: $6$, $15$, $18$, $24$, $30$,
$33$, $42$, $45$, $51$, $54$, $60$, $66$, $69$, $72$, \abk $78$,
$87$, \abk $90$, \abk $96$, $99$. \item[{\rm (ii)}] $k = 5$,
$n\leq 25$: $10$, $15$.
\end{itemize}
Some of these orders are covered by general results (see
Remark~\ref{deL84Cor}).
\end{remark}

Henceforth $p$ is odd.
\begin{lemma}\label{PartialIGU}
Let $k=p^t$ and $n=p^rm$ where $p\nmid m$. Suppose that $H$ is a
cocyclic $\BH(n,k)$ with indexing group $G$ such that $G/G'$ has a
cyclic subgroup of order $p^r$. Then any cocycle $\psi\in
I(G,\Cr_{k})$ of $H$ is in $I(G,\Cr_{k})^p$.
\end{lemma}
\begin{proof} (Cf.~\cite[Corollary~7.44]{Horadam}.)
By Subsection~\ref{CCocycles}, $\psi = \psi_1\partial \phi$ for
some $\psi_1$ inflated from $Z(G/G',\Cr_{k})$ and map $\phi$.
Assume that $\psi_1\not \in I(G,\Cr_{k})^p$. Then
$[\psi_1(x,y)]_{x,y\in G}$ has a row with $m$ occurrences of
$\zeta_k$ and every other entry equal to $1$. Label this row $a$.
Now
\begin{align*}
{\textstyle \prod}_{y\in G} \partial \phi (a,y) &=
\big({\textstyle \prod}_{y\in G} \phi (a)^{-1} \big) \big(
{\textstyle \prod}_{y\in G} \phi (y)^{-1} \big) \big( {\textstyle
\prod}_{y\in G} \phi
(ay)\big)\\
&= \phi(a)^{-n} \in \langle \zeta_k^p\rangle.
\end{align*}
So, if we multiply along row $a$ of $[\psi(x,y)]_{x,y\in G}$ then
we get an element of $\langle \zeta_k\rangle \setminus \langle
\zeta_k^p\rangle$. But this is a contradiction. For suppose that
$\sum_{i=0}^{k-1}c_i{\zeta_k}^{\! i}=0$. Since the $k$th
cyclotomic polynomial $\sum_{i=0}^{p-1}{\rm x}^{i( p^{t-1})}$
divides $\sum_{i=0}^{k-1}c_i{\rm x}^{i}$, we have $c_j =
c_{p^{t-1}+j} = \cdots = c_{(p-1)p^{t-1}+j}$, $0\leq j\leq
p^{t-1}-1$. It is then straightforward to verify that ${\textstyle
\prod}_{i=0}^{k-1}\zeta_k^{i c_i} \in \langle \zeta_k^p\rangle$.
\qedhere
\end{proof}

\begin{corollary}\label{pSquareFree}
If $n$ is $p$-square-free then a cocyclic $\BH(n,p)$ is equivalent
to a group-developed matrix.
\end{corollary}
\begin{proof}
Let $G$ be the indexing group of a cocyclic $\BH(n,p)$. Either $p$
divides $|G'|$ or Lemma~\ref{PartialIGU} applies, and thus
$I(G,\Cr_{p})=B(G,\Cr_p)$. Also $\mathrm{Hom}(H_2(G),\Cr_{p})=1$
by \cite[Theorem~2.1.5]{Karpilovsky}. \qedhere
\end{proof}

Proposition~\ref{CircppFourier} then yields
\begin{corollary}\label{CBHppisFourier}
A cocyclic $\BH(p,p)$ is equivalent to the Fourier matrix of order
$p$.
\end{corollary}
\begin{remark}\label{ButsonNotRemark}
By Remark~\ref{GDExclusions} and Corollary~\ref{pSquareFree}, for
$(n,p)=(10,5)$ or $p=3$ and $n\in \abk \{ 6,24,30\}$, there are no
cocyclic $\BH(n,p)$ at all (so Butson's construction~\cite{Butson}
is not cocyclic). Furthermore, a cocyclic $\BH(12,3)$,
$\BH(21,3)$, $\BH(20,5)$, or $\BH(14,7)$ is equivalent to a
group-developed matrix.
\end{remark}

\subsection{Existence of cocyclic $\BH(n,p)$, $np\leq 100$}
\label{SomeNonExResults}

The table below summarizes existence of matrices in our
classification.
\begin{table}[h]
\centering
\begin{tabular}{|c|c|c|c|c|c|c|c|c|c|c|c|} \hline
$p$ $\backslash$ $\frac{n}{p}$ & \phantom{1}$1$\phantom{1} &
\phantom{1}$2$\phantom{1} & \phantom{1}$3$\phantom{1} &
\phantom{1}$4$\phantom{1} & \phantom{1}$5$\phantom{1} &
\phantom{1}$6$\phantom{1} & \phantom{1}$7$\phantom{1} &
\phantom{1}$8$\phantom{1} & \phantom{1}$9$\phantom{1} &
\phantom{1}$10$\phantom{1} & \phantom{1}$11$\phantom{1} \\ \hline
$3$ & F & NC & E& E & N & S$_2$ & S$_1$ & NC & E & NC & N \\
$5$ & F & NC & N & S$_1$ &  &   &   &   &   &   &   \\
$7$ & F & S$_1$ & &  &  &   &   &   &   &   &   \\
\hline
\end{tabular}
\label{TableA}\caption{Existence of $\BH(n,p)$}
\end{table}

\vspace{-10pt}

{\small \noindent
N: no Butson Hadamard matrices by Remark~\ref{deL84Cor}. \\
NC: no cocyclic Butson Hadamard matrices
by Remark~\ref{ButsonNotRemark}. \\
E: cocyclic Butson Hadamard matrices exist.
See Section~\ref{FullClassification}.\\
S$_1$: no cocyclic Butson Hadamard matrices according
to a relative difference set search. \\
S$_2$: no cocyclic Butson Hadamard matrices
according to an orthogonal cocycle search.\\
F: the Fourier matrix is the only Butson Hadamard matrix by
Proposition~\ref{HiraSmall} (or Corollary~\ref{CBHppisFourier}). }

\begin{remark} There are non-cocyclic $\BH(6,3)$ and
$\BH(10,5)$ by \cite{Butson}. Non-existence of cocyclic $\BH(6,3)$
is established by computer in \cite[Example 7.4.2]{Horadam}.
\end{remark}

We relied on computation of relative difference sets only for
parameter values that we could not settle otherwise. Nevertheless,
those calculations were not onerous. The search for a relative
difference set with parameters $(14, 7, 14, 2)$ ran in under an
hour; the test for an $\mathrm{RDS}(20, 5, 20, 4)$ took about a
day, with most of the time being spent on $\Cr_{100}$. We note
additionally that there are theoretical obstructions to the
existence of an $\mathrm{RDS}(21, 3, 21, 7)$: the system of
diophantine \textit{signature equations} that such a difference
set must satisfy does not admit a solution~\cite{RoederDiss}.

\section{The full classification}\label{FullClassification}

The only cases left to deal with are $(n,p)\! \in \! \{ (9,3),
(12,3), (27,3)\}$. In this section we discuss our complete and
irredundant classification of such $\BH(n,p)$.

Our overall task splits into two steps. We first compute a set of
cocyclic $\BH(n,p)$ containing representatives of every
equivalence class. Then we test equivalence of the matrices
produced. Since our method for the second step was given in
Section~\ref{EquivalenceTesting}, and the orders involved pose no
computational difficulties, we say nothing further about this
step. Two complementary methods were used for the first step:
checking shift orbits for orthogonal cocycles, and constructing
relative difference sets. See Subsections~\ref{CCocycles} and
\ref{ShiftAction}; also, we refer to \cite[Section~6]{POCRod},
which discusses a classification of cocyclic Hadamard matrices via
central relative difference sets. The algorithm for constructing
the difference sets in this paper is identical to the one there,
and was likewise carried out using M.~R\"{o}der's  \textsf{GAP}
package \textit{RDS}~\cite{RDS}.

\begin{example}\label{FullEnumeration}
Table~\ref{Table2} lists the number $t$ of orthogonal elements of
$Z(G,\Cr_3)$ for $|G|=9$ or $12$.

\begin{table}[H]
\centering \begin{tabular}{| c | c | c | c | c | c | c |c|} \hline
\phantom{1}$G$\phantom{1}  & \phantom{1}$\Cr_{9}$\phantom{1} &
\phantom{1}$\Cr_{3}^{2}$\phantom{1} &
\phantom{1}$\Cr_{12}$\phantom{1} & \phantom{1}$\Cr_{3} \rtimes
\Cr_{4}$\phantom{1} & \phantom{1}$\mathrm{Alt}(4)$\phantom{1} &
\phantom{1}$\D_{6}$\phantom{1} & \phantom{1}$\Cr_{2}^{2} \times
\Cr_{3}$\phantom{1}
\\ \hline
{\small $t$} & 18 & 144 & 0 & 288 & 48 & 0 & 96 \\
\hline
\end{tabular}
\caption{Counting orthogonal elements of $Z(G,\Cr_3)$}
\label{Table2}
\end{table}

\vspace{-5pt}

\noindent If $|G|\in \{6, 15, 18\}$ then $t=0$.
\end{example}

\subsection{$\BH(9,3)$.} There are precisely three equivalence
classes of cocyclic $\BH(9,3)$.

One class contains $\BH(3,3)\otimes \BH(3,3)$, which has indexing
group $\Cr_3^2$ and cocycle that is not a coboundary. Some
matrices $H_1$ in this class are group-developed over $\Cr_3^2$.
No $H_1$ has indexing group $\Cr_9$. See Examples~\ref{Drakeetc}
and \ref{DrakeetcContinued}.

Another equivalence class contains group-developed matrices with
indexing group $\Cr_9$. No matrix $H_2$ in this class has indexing
group $\Cr_3^2$; hence the cocycles of $H_2$ are all coboundaries
by Lemma~\ref{PartialIGU}. This class is not represented in
\cite{CHLibrary}, but happens to be an example of the construction
in \cite{deLCircGH} (cf.~\cite{Brock}). A representative is the
circulant with first row $( 1 , 1 , 1 , 1 , \zeta_3 ,\abk
\zeta_3^2 , 1 , \zeta_3^2 , \zeta_3 )$.

The third class contains matrices $H_3\approx H_2^*$ that are
cocyclic with indexing group $\Cr_9$. Again, $H_3$ is equivalent
to a circulant, does not have indexing group $\Cr_3^2$, all of its
cocycles are coboundaries, and  it is not in \cite{CHLibrary}.

By Proposition~\ref{LambdaEquivalenceImplies},
$\mathrm{PAut}(\mathcal{E}_{H_2}) \cong
\mathrm{PAut}(\mathcal{E}_{H_3})$, and this is solvable. We
described $\mathrm{PAut}(\mathcal{E}_{H_1})$ in
Example~\ref{Drakeetc}.

\subsection{$\BH(12,3)$.} Each cocyclic
$\BH(12, 3)$ is equivalent to a group-developed matrix
(Remark~\ref{ButsonNotRemark}) over one of $\Cr_{3} \rtimes \abk
\Cr_{4}$, $\Cr_{2}^{2} \rtimes \Cr_{3}$, or $\Cr_{2}^{2} \times
\Cr_{3}$. There are just two equivalence classes, which form a
Hermitian pair. The automorphism groups have order $864$.

This is the only order $n$ in our classification which is not a
prime power and for which cocyclic $\BH(n,p)$ exist.

\subsection{$\BH(27,3)$.} Predictably, order $27$ was the most
challenging one that we faced in our computations. An exhaustive
search for orthogonal cocycles was not possible, so this order was
classified by the central relative difference sets method.

There are sixteen equivalence classes of cocyclic $\BH(27, 3)$ in
total. Some are Kronecker products of cocyclic $\BH(9, 3)$ with
the unique $\BH(3,3)$, but the majority are not of this form. Each
matrix is equivalent to its transpose. There are two classes that
are self-equivalent under the Hermitian; the rest occur in
distinct Hermitian pairs.

Except for the generalized Sylvester matrix, whose automorphism
group as stated in Example~\ref{Drakeetc} is not solvable, the
automorphism group of a $\BH(27,3)$ has order $2^{a}3^{b}$.

Every non-cyclic group of order $27$ is an indexing group of at
least one $\BH(27, 3)$. There are no circulants.

\section{Concluding comments}\label{Concluding}

It is noteworthy that all matrices in our classification are
equivalent to group-developed ones (non-trivial cohomology classes
appear too). This may be compared with \cite{POCRod}, which
features many equivalence classes not containing group-developed
Hadamard matrices. Also, while there exist circulant $\BH(p^r,p)$
for all odd $p$ and $r\leq 2$~\cite{Brock,deLCircGH}, we have not
yet found a circulant $\BH(n,p)$ when $n$ is not a $p$-power.

A few composition results should be given. Let $\psi_i\in
Z(G_i,\Cr_k)$ for $i=\abk 1$, $2$, and define $\psi\in Z(G_1\times
G_2,\Cr_k)$ by $\psi((a,b),(x,y)) = \psi_1(a,x)\psi_2(b,y)$. It is
not hard to show that $\psi \in B(G_1\times G_2, \Cr_k)$ if and
only if $\psi_1$, $\psi_2$ are coboundaries.
\begin{lemma}\label{InevitableKronecker}
Suppose that $H_i$ is a cocyclic $\BH(n_i,k)$ with cocycle
$\psi_i$, $1\leq i \leq 2$. Then $H_1\otimes H_2$ is a cocyclic
$\BH(n_1n_2,k)$ with cocycle $\psi$.
\end{lemma}

\begin{corollary}\label{LikedeL92}
For $a\geq 1$, $b\geq a$, and $G \in \{ \Cr_{3} \rtimes \Cr_{4},
\Cr_{2}^{2} \rtimes \Cr_{3}, \Cr_{2}^{2} \times \Cr_{3} \}$, there
exists a group-developed $\BH(2^{2a}3^b,3)$ with indexing group
$G^a\times \Cr_{3}^{b-a}$.
\end{corollary}

Corollary~\ref{LikedeL92} was proved by de
Launey~\cite[Corollary~3.10]{deL92}, albeit only for indexing
groups $\Cr_2^{2a}\times \Cr_3^b$.

\subsubsection*{Acknowledgments}
R.~Egan received funding from the Irish Research Council
(Government of Ireland Postgraduate Scholarship). P. \'{O}
Cath\'{a}in was supported by Australian Research Council grant
DP120103067.

\end{document}